\documentclass[10pt]{article}

\usepackage{amsmath,amsbsy,amssymb,latexsym,amsthm,dsfont}
\usepackage{a4wide}
\usepackage{graphicx}
\usepackage{cite}

\newtheorem{theorem}{Theorem}
\newtheorem{corollary}[theorem]{Corollary}
\newtheorem{lemma}[theorem]{Lemma}
\newtheorem{definition}[theorem]{Definition}
\newtheorem{remark}[theorem]{Remark}
\newtheorem{example}[theorem]{Example}

\newenvironment{keywords}{\begin{center}
\begin{minipage}[c]{13.4cm} {\bf Keywords:}} {\end{minipage}
\end{center}}

\newenvironment{msc}{\begin{center}
\begin{minipage}[c]{13.4cm} {\bf MSC 2010:}} {\end{minipage}
\end{center}}

\begin{document}

\title{Fractional variational problems\\
depending on indefinite integrals\thanks{Part
of the second author's Ph.D., which is carried out
at the University of Aveiro under the
Doctoral Program in \emph{Mathematics and Applications}
(PDMA) of Universities of Aveiro and Minho.
Submitted 29-Dec-2010; revised 14-Feb-2011; accepted 16-Feb-2011; for publication
in \emph{Nonlinear Analysis Series A: Theory, Methods \& Applications}.}}

\author{Ricardo Almeida\\
\texttt{ricardo.almeida@ua.pt}
\and Shakoor Pooseh\\
\texttt{spooseh@ua.pt}
\and Delfim F. M. Torres\\
\texttt{delfim@ua.pt}}

\date{Department of Mathematics,
University of Aveiro,
3810-193 Aveiro, Portugal}

\maketitle


\begin{abstract}
We obtain necessary optimality conditions
for variational problems with a Lagrangian
depending on a Caputo fractional derivative, a fractional and
an indefinite integral. Main results give fractional
Euler--Lagrange type equations and natural boundary conditions,
which provide a generalization of previous results found in the literature.
Isoperimetric problems, problems with holonomic constraints and
depending on higher-order Caputo derivatives,
as well as fractional Lagrange problems, are considered.
\end{abstract}

\begin{msc}
49K05, 49S05, 26A33, 34A08.
\end{msc}



\section{Introduction}

In the 18th century, Euler considered the problem of optimizing functionals
depending not only on some unknown function $y$ and some derivative of $y$,
but also on an antiderivative of $y$ (see \cite{fraser}). Similar problems have been
recently investigated in \cite{Gregory}, where Lagrangians containing
higher-order derivatives and optimal control problems are considered.
More generally, it has been shown that the results of \cite{Gregory}
hold on an arbitrary time scale \cite{Nat}.
Here we study such problems within the framework of fractional calculus.

Roughly speaking, a fractional calculus defines integrals
and derivatives of non-integer order.
Let $\alpha>0$ be a real number
and $n\in\mathbb{N}$ be such that $n-1<\alpha<n$.
Here we follow \cite{Almeida1} and \cite{Kilbas,Miller}.
Let $f:[a,b]\to\mathbb{R}$ be piecewise continuous
on $(a,b)$ and integrable on $[a,b]$.
The left and right Riemann--Liouville fractional integrals of $f$
of order $\alpha$ are defined respectively by
$$
{_aI_x^\alpha}f(x)=\frac{1}{\Gamma(\alpha)}\int_a^x (x-t)^{\alpha-1}f(t)dt
\quad \mbox{and} \quad
{_xI_b^\alpha}f(x)=\frac{1}{\Gamma(\alpha)}\int_x^b(t-x)^{\alpha-1} f(t)dt.
$$
Here $\Gamma$ is the well-known Gamma function. Then
the left ${_aD_x^\alpha}$ and right ${_xD_b^\alpha}$
Riemann--Liouville fractional derivatives of $f$
of order $\alpha$ are defined (if they exist) as
\begin{equation}
\label{ded:RL:left}
{_aD_x^\alpha}f(x)=\frac{1}{\Gamma(n-\alpha)}\frac{d^n}{dx^n}
\int_a^x(x-t)^{n-\alpha-1}f(t)dt
\end{equation}
and
\begin{equation}
\label{ded:RL:right}
{_xD_b^\alpha}f(x)=\frac{(-1)^n}{\Gamma(n-\alpha)}\frac{d^n}{dx^n}\int_x^b (t-x)^{n-\alpha-1} f(t)dt.
\end{equation}
The fractional derivatives \eqref{ded:RL:left} and \eqref{ded:RL:right}
have one disadvantage when modeling real world phenomena:
the fractional derivative of a constant is not zero.
To eliminate this problem, one often
considers fractional derivatives in the sense of Caputo.
Let $f$ belong to the space $AC^n([a,b];\mathbb{R})$
of absolutely continuous functions.
The left and right Caputo fractional
derivatives of $f$ of order $\alpha$ are defined respectively by
$$
{_a^CD_x^\alpha}f(x)
=\frac{1}{\Gamma(n-\alpha)}\int_a^x (x-t)^{n-\alpha-1}f^{(n)}(t)dt
$$
and
$$
{_x^CD_b^\alpha}f(x)
=\frac{1}{\Gamma(n-\alpha)}\int_x^b(-1)^n(t-x)^{n-\alpha-1} f^{(n)}(t)dt.
$$
These fractional integrals and derivatives define a rich calculus.
For details see the books \cite{Kilbas,Miller,samko}. Here we just recall
a useful property for our purposes: integration by parts.
For fractional integrals,
$$
\displaystyle\int_{a}^{b}  g(x) \cdot {_aI_x^\alpha}f(x)dx
=\int_a^b f(x) \cdot {_x I_b^\alpha} g(x)dx
$$
(see, \textrm{e.g.}, \cite[Lemma~2.7]{Kilbas}),
and for Caputo fractional derivatives
$$
\displaystyle\int_{a}^{b} g(x)\cdot {_a^C D_x^\alpha}f(x)dx
=\displaystyle\int_a^b f(x)\cdot {_x D_b^\alpha} g(x)dx+\sum_{j=0}^{n-1}
\left[{_xD_b^{\alpha+j-n}}g(x) \cdot f^{(n-1-j)}(x)\right]_a^b
$$
(see, \textrm{e.g.}, \cite[Eq. (16)]{MR2345467}).
In particular, for $\alpha\in(0,1)$ one has
\begin{equation}
\label{eq:frac:IBP}
\int_{a}^{b}g(x)\cdot {_a^C D_x^\alpha}f(x)dx
=\int_a^b f(x)\cdot {_x D_b^\alpha} g(x)dx
+\left[{_xI_b^{1-\alpha}}g(x) \cdot f(x)\right]_a^b.
\end{equation}
When $\alpha \rightarrow 1$, ${_a^C D_x^\alpha} = \frac{d}{dx}$,
${_x D_b^\alpha} = - \frac{d}{dx}$, ${_xI_b^{1-\alpha}}$
is the identity operator, and \eqref{eq:frac:IBP} gives
the classical formula of integration by parts.

The fractional calculus of variations concerns finding extremizers
for variational functionals depending on fractional derivatives
instead of integer ones. The theory started in 1996 with
the work of Riewe, in order to better describe non-conservative systems
in mechanics \cite{CD:Riewe:1996,CD:Riewe:1997}.
The subject is now under strong development due
to its many applications in physics and engineering,
providing more accurate models of physical phenomena (see, \textrm{e.g.},
\cite{MyID:182,MyID:154,MyID:179,El-Nabulsi1,El-Nabulsi2,MyID:191,gastao:delfim,gasta1,MyID:181,MyID:203}).
With respect to results on fractional variational calculus via Caputo operators,
we refer the reader to \cite{AGRA1,Ankara:Ric,Almeida,Gastao,Malinowska,MyID:163,MyID:207}
and references therein.

Our main contribution is an extension of the results presented
in \cite{AGRA1,Gregory} by considering Lagrangians containing an antiderivative,
that in turn depend on the unknown function, a fractional integral, and a Caputo
fractional derivative (Section~\ref{sec:Fundprob}).
Transversality conditions are studied in Section~\ref{sec:natbound},
where the variational functional $J$ depends also on the terminal time $T$,
$J(y,T)$, and where we obtain conditions for a pair $(y,T)$ to be an optimal
solution to the problem. In Section~\ref{sec:IsoProb}
we consider isoperimetric problems with integral constraints
of the same type as the cost functionals considered in Section~\ref{sec:Fundprob}.
Fractional problems with holonomic constraints are considered in Section~\ref{sec:Holonomic}.
The situation when the Lagrangian depends on higher
order Caputo derivatives, \textrm{i.e.}, it depends on
$^C_aD_x^{\alpha_k}y(x)$ for $\alpha_k\in(k-1,k)$, $k\in\{1,\ldots,n\}$,
is studied in Section~\ref{sec:Higher}, while Section~\ref{sec:FracOpt}
considers fractional Lagrange problems and the Hamiltonian
approach. In Section~\ref{sec:SufConditions}
we obtain sufficient conditions of optimization
under suitable convexity assumptions
on the Lagrangian. We end with Section~\ref{sec:NumSim},
discussing a numerical scheme for solving
the proposed fractional variational problems.
The idea is to approximate fractional problems
by classical ones. Numerical results for
two illustrative examples are described in detail.


\section{The fundamental problem}
\label{sec:Fundprob}

Let $\alpha\in(0,1)$ and $\beta>0$.
The problem that we address is stated in the following way.
Minimize the cost functional
\begin{equation}
\label{funct}
J(y)=\int_a^b L(x,y(x),{^C_aD_x^\alpha}y(x),{_aI_x^\beta}y(x),z(x))dx,
\end{equation}
where the variable $z$ is defined by
$$
z(x)=\int_a^x l(t,y(t),{^C_aD_t^\alpha}y(t),{_aI_t^\beta}y(t))dt,
$$
subject to the boundary conditions
\begin{equation}
\label{bound}
y(a)=y_a \quad \mbox{and} \quad y(b)=y_b.
\end{equation}
We assume that the functions $(x,y,v,w,z)\to L(x,y,v,w,z)$
and  $(x,y,v,w)\to l(x,y,v,w)$ are of class $C^1$, and the trajectories
$y:[a,b]\to\mathbb{R}$ are absolute continuous functions, $y \in AC([a,b];\mathbb{R})$,
such that ${^C_aD_x^\alpha}y(x)$ and ${_aI_x^\beta}y(x)$ exist and are continuous on $[a,b]$.
We denote such class of functions by $\mathcal{F}([a,b];\mathbb{R})$.
Also, to simplify, by $[\cdot]$ and $\{\cdot\}$ we denote the operators
$$
[y](x)=(x,y(x),{^C_aD_x^\alpha}y(x),{_aI_x^\beta}y(x),z(x))
\quad \mbox{and}\quad \{y\}(x)=(x,y(x),{^C_aD_x^\alpha}y(x),{_aI_x^\beta}y(x)).
$$

\begin{theorem}
\label{ELTEo}
Let $y \in \mathcal{F}([a,b];\mathbb{R})$ be a minimizer
of $J$ as in \eqref{funct},
subject to the boundary conditions \eqref{bound}.
Then, for all $x\in[a,b]$, $y$ is a solution
of the fractional equation
\begin{multline}
\label{ELeq}
\frac{\partial L}{\partial y}[y](x)
+{_xD^\alpha_b}\left( \frac{\partial L}{\partial v}[y](x) \right)
+{_xI_b^\beta}\left(\frac{\partial L}{\partial w}[y](x)\right)
+\int_x^b \frac{\partial L}{\partial z}[y](t)dt\cdot \frac{\partial l}{\partial y}\{y\}(x)\\
+{_xD^\alpha_b}\left( \int_x^b \frac{\partial L}{\partial z}[y](t)dt
\cdot \frac{\partial l}{\partial v}\{y\}(x)  \right)
+{_xI^\beta_b}\left( \int_x^b \frac{\partial L}{\partial z}[y](t)dt
\cdot \frac{\partial l}{\partial w}\{y\}(x)  \right)=0.
\end{multline}
\end{theorem}

\begin{proof}
Let $h\in \mathcal{F}([a,b];\mathbb{R})$
be such that $h(a)=0=h(b)$, and $\epsilon$
be a real number with $|\epsilon| \ll 1$. If we define $j$ as
$j(\epsilon)=J(y+\epsilon h)$, then $j'(0)=0$.
Differentiating $j$ at $\epsilon=0$, we get
\begin{multline*}
\int_a^b \left[ \frac{\partial L}{\partial y}[y](x)h(x)
+ \frac{\partial L}{\partial v}[y](x){^C_aD^\alpha_x}h(x)
+ \frac{\partial L}{\partial w}[y](x){_aI^\beta_x}h(x)\right.\\
\left.+\frac{\partial L}{\partial z}[y](x)\int_a^x\left(
\frac{\partial l}{\partial y}\{y\}(t)h(t)
+\frac{\partial l}{\partial v}\{y\}(t){^C_aD^\alpha_t}h(t)
+\frac{\partial l}{\partial w}\{y\}(t){_aI^\beta_t}h(t)\right)dt\right]dx=0.
\end{multline*}
The necessary condition \eqref{ELeq} follows
from the next relations and the fundamental lemma
of the calculus of variations
(\textrm{cf.}, \textrm{e.g.}, \cite[p.~32]{Brunt}):
$$
\int_a^b \frac{\partial L}{\partial v}[y](x){^C_aD^\alpha_x}h(x) dx
=\int_a^b {_x D_b^\alpha} \left(\frac{\partial L}{\partial v}[y](x) \right)h(x)dx
+ \left[{_xI_b^{1-\alpha}}\left(\frac{\partial L}{\partial v}[y](x) \right) h(x)\right]_a^b,
$$
$$
\int_a^b \frac{\partial L}{\partial w}[y](x){_aI^\beta_x}h(x) dx
=\int_a^b {_x I_b^\beta} \left(\frac{\partial L}{\partial w}[y](x) \right)h(x)dx,
$$
\begin{align*}
\int_a^b & \frac{\partial L}{\partial z}[y](x)\left(\int_a^x
\frac{\partial l}{\partial y}\{y\}(t)h(t) dt \right) dx
=\int_a^b \left( -\frac{d}{dx}\int_x^b\frac{\partial L}{\partial z}[y](t)dt \right)
\left( \int_a^x \frac{\partial l}{\partial y}\{y\}(t)h(t) dt \right) dx\\
&= \left[- \left(\int_x^b\frac{\partial L}{\partial z}[y](t)dt \right)\left(\int_a^x
\frac{\partial l}{\partial y}\{y\}(t)h(t) dt\right)  \right]_a^b
+\int_a^b \left(\int_x^b\frac{\partial L}{\partial z}[y](t)dt \right)
\frac{\partial l}{\partial y}\{y\}(x)h(x) \, dx\\
&= \int_a^b \left(\int_x^b\frac{\partial L}{\partial z}[y](t)dt \right)
\frac{\partial l}{\partial y}\{y\}(x)h(x) \, dx,
\end{align*}
\begin{align*}
\int_a^b & \frac{\partial L}{\partial z}[y](x)\left(\int_a^x
\frac{\partial l}{\partial v}\{y\}(t){^C_aD^\alpha_t}h(t) dt \right) dx
= \int_a^b \left( -\frac{d}{dx}\int_x^b\frac{\partial L}{\partial z}[y](t)dt \right)
\left( \int_a^x \frac{\partial l}{\partial v}\{y\}(t){^C_aD^\alpha_t}h(t) dt \right)  dx\\
&= \left[- \left(\int_x^b\frac{\partial L}{\partial z}[y](t)dt \right)\left(\int_a^x
\frac{\partial l}{\partial v}\{y\}(t){^C_aD^\alpha_t}h(t) dt\right) \right]_a^b
+\int_a^b \left(\int_x^b\frac{\partial L}{\partial z}[y](t)dt \right)
\frac{\partial l}{\partial v}\{y\}(x){^C_aD^\alpha_x}h(x) \,  dx\\
&= \int_a^b {_xD^\alpha_b}\left(\int_x^b\frac{\partial L}{\partial z}[y](t)dt
\frac{\partial l}{\partial v}\{y\}(x)\right)h(x) \, dx
+\left[ {_xI^{1-\alpha}_b}\left(\int_x^b\frac{\partial L}{\partial z}[y](t)dt
\frac{\partial l}{\partial v}\{y\}(x) \right) h(x) \right]_a^b,
\end{align*}
and
$$
\int_a^b \frac{\partial L}{\partial z}[y](x)\left(\int_a^x
\frac{\partial l}{\partial w}\{y\}(t){_aI^\beta_t}h(t) dt \right) dx
=\int_a^b {_xI^\beta_b}\left(\int_x^b\frac{\partial L}{\partial z}[y](t)dt
\frac{\partial l}{\partial w}\{y\}(x)\right)h(x) \, dx.
$$
\end{proof}

The fractional Euler--Lagrange equation \eqref{ELeq} involves
not only fractional integrals and fractional derivatives,
but also indefinite integrals. Theorem~\ref{ELTEo} gives
a necessary condition to determine the possible choices
for extremizers.

\begin{definition}
Solutions to the fractional Euler--Lagrange equation \eqref{ELeq}
are called extremals for $J$ defined by \eqref{funct}.
\end{definition}

\begin{example}
\label{example:bra}
Consider the functional
\begin{equation}
\label{example}
J(y)=\int_0^1 \left[({^C_0D_x^\alpha}y(x)-\Gamma(\alpha+2)x)^2+z(x)\right]dx,
\end{equation}
where $\alpha \in (0,1)$ and
$$
z(x)=\int_0^x (y(t)-t^{\alpha+1})^2 \, dt,
$$
defined on the set
$$
\left\{ y\in \mathcal{F}([0,1];\mathbb{R})
:\, y(0)=0 \, \mbox{ and }\,  y(1)=1\right\}.
$$
Let
\begin{equation}
\label{eq:ext:ne}
y_\alpha(x)=x^{\alpha+1},\quad x\in[0,1].
\end{equation}
Then,
$$
{^C_0D_x^\alpha}y_\alpha(x)=\Gamma(\alpha+2)x.
$$
Since $J(y)\geq0$ for all admissible functions $y$,
and $J(y_\alpha)=0$, we have that $y_\alpha$ is a minimizer of $J$.
The Euler--Lagrange equation applied to \eqref{example} gives
\begin{equation}
\label{eq:fELeq:ne}
{_xD_1^\alpha}({^C_0D_x^\alpha}y(x)-\Gamma(\alpha+2)x)
+\int_x^1 1dt \, (y(x)-x^{\alpha+1})=0.
\end{equation}
Obviously, $y_\alpha$ is a solution of the fractional
differential equation \eqref{eq:fELeq:ne}.
\end{example}

The extremizer \eqref{eq:ext:ne}
of Example~\ref{example:bra} is smooth on the
closed interval $[0,1]$. This is not always
the case. As next example shows, minimizers
of \eqref{funct}--\eqref{bound} are not necessarily
$C^1$ functions.

\begin{example}
\label{example1}
Consider the following fractional variational problem: to minimize the functional
\begin{equation}
\label{funcExample}
J(y)=\int_0^1\left[ \left({^C_0D_x^\alpha}y(x)-1\right)^2+z(x)\right]dx
\end{equation}
on
$$
\left\{ y \in \mathcal{F}([0,1];\mathbb{R})
\, : \, y(0)=0 \quad \mbox{and}
\quad y(1)=\frac{1}{\Gamma(\alpha+1)}\right\},
$$
where $z$ is given by
$$
z(x)=\int_0^x \left( y(t)-\frac{t^\alpha}{\Gamma(\alpha+1)}\right)^2dt.
$$
Since ${^C_0D_x^\alpha}x^\alpha=\Gamma(\alpha+1)$,
we deduce easily that function
\begin{equation}
\label{eq:gs:ex}
\overline{y}(x)=\frac{x^\alpha}{\Gamma(\alpha+1)}
\end{equation}
is the global minimizer to the problem. Indeed, $J(y)\geq 0$
for all $y$, and $J(\overline y)=0$. Let us see that
$\overline{y}$ is an extremal for $J$.
The fractional Euler--Lagrange equation \eqref{ELeq} becomes
\begin{equation}
\label{ELeqExample}
2 \, {_xD_1^\alpha} ({^C_0D_x^\alpha}y(x)-1) + \int_x^1 1\, dt
\cdot  2 \left( y(x)-\frac{x^\alpha}{\Gamma(\alpha+1)}\right)=0.
\end{equation}
Obviously, $\overline{y}$ is a solution of equation \eqref{ELeqExample}.
\end{example}

\begin{remark}
The minimizer \eqref{eq:gs:ex} of Example~\ref{example1}
is not differentiable at 0, as $0<\alpha<1$. However,
$\overline{y}(0)=0$ and  ${^C_0D_x^\alpha}\overline{y}(x)
={_0D_x^\alpha}\overline{y}(x)=\Gamma(\alpha+1)$ for any
$x \in [0,1]$.
\end{remark}

\begin{corollary}[\textrm{cf.} equation (9) of \cite{AGRA1}]
If $y$ is a minimizer of
\begin{equation}
\label{funct7}
J(y)=\int_a^b L(x,y(x),{^C_aD_x^\alpha}y(x))dx,
\end{equation}
subject to the boundary conditions \eqref{bound},
then $y$ is a solution of the fractional equation
$$
\frac{\partial L}{\partial y}[y](x)
+{_xD^\alpha_b}\left( \frac{\partial L}{\partial v}[y](x) \right)=0.
$$
\end{corollary}

\begin{proof}
Follows from Theorem~\ref{ELTEo} with an $L$ that
does not depend on ${_aI_x^\beta}y$ and $z$.
\end{proof}

We now derive the Euler--Lagrange equations for functionals
containing several dependent variables,
\textrm{i.e.}, for functionals of type
\begin{equation}
\label{functMul}
J(y_1,\ldots,y_n)
=\int_a^b L(x,y_1(x),\ldots, y_n(x),
{^C_aD_x^\alpha}y_1(x),\ldots,{^C_aD_x^\alpha}y_n(x),
{_aI_x^\beta}y_1(x),\ldots,{_aI_x^\beta}y_n(x),z(x))dx,
\end{equation}
where $n\in \mathbb{N}$ and $z$ is defined by
$$
z(x)=\int_a^x l(t,y_1(t),\ldots,y_n(t),{^C_aD_t^\alpha}y_1(t),
\ldots,{^C_aD_t^\alpha}y_n(t),{_aI_t^\beta}y_1(t),\ldots,{_aI_t^\beta}y_n(t))dt,
$$
subject to the boundary conditions
\begin{equation}
\label{boundMul}
y_k(a)=y_{a,k} \quad \mbox{and}
\quad y_k(b)=y_{b,k},\quad k\in\{1,\ldots,n\}.
\end{equation}
To simplify, we consider $y$ as the vector
$y=(y_1,\ldots,y_n)$.
Consider a family of variations $y+\epsilon h$, where
$|\epsilon| \ll 1$ and $h=(h_1,\ldots,h_n)$.
The boundary conditions \eqref{boundMul} imply that
$h_k(a)=0=h_k(b)$, for $k\in\{1,\ldots,n\}$.
The following theorem can be easily proved.

\begin{theorem}
\label{ELeqMult}
Let $y$ be a minimizer of $J$ as in \eqref{functMul},
subject to the boundary conditions \eqref{boundMul}.
Then, for all $k\in\{1,\ldots,n\}$ and for all $x\in[a,b]$,
$y$ is a solution of the fractional Euler--Lagrange equation
\begin{multline*}
\frac{\partial L}{\partial y_k}[y](x)+{_xD^\alpha_b}\left( \frac{\partial L}{\partial v_k}[y](x) \right)+
{_xI^\beta_b}\left( \frac{\partial L}{\partial w_k}[y](x) \right)+\int_x^b \frac{\partial L}{\partial z}[y](t)dt\cdot
\frac{\partial l}{\partial y_k}\{y\}(x)\\
+{_xD^\alpha_b}\left( \int_x^b \frac{\partial L}{\partial z}[y](t)dt\cdot \frac{\partial l}{\partial v_k}\{y\}(x)  \right)
+{_xI^\beta_b}\left( \int_x^b \frac{\partial L}{\partial z}[y](t)dt\cdot \frac{\partial l}{\partial w_k}\{y\}(x)  \right)=0.
\end{multline*}
\end{theorem}


\section{Natural boundary conditions}
\label{sec:natbound}

In this section we consider a more general question.
Not only the unknown function $y$
is a variable in the problem,
but also the terminal time $T$ is an unknown.
For $T\in[a,b]$, consider the functional
\begin{equation}
\label{funct4}
J(y,T)=\int_a^T L[y](x)dx,
\end{equation}
where
$$
[y](x)=(x,y(x),{^C_aD_x^\alpha}y(x),{_aI_x^\beta}y(x),z(x)).
$$
The problem consists in finding a pair
$(y,T)\in \mathcal{F}([a,b];\mathbb{R})\times [a,b]$
for which the functional $J$ attains a minimum value.
First we give a remark that will be used later
in the proof of Theorem~\ref{TeoNatural}.

\begin{remark}
\label{remarkIntegral}
If $\phi$ is a continuous function, then (\textrm{cf.} \cite[p.~46]{Miller})
$$
\lim_{x\to T}{_xI_T^{1-\alpha}}\phi(x)=0
$$
for any $\alpha \in (0,1)$.
\end{remark}

\begin{theorem}
\label{TeoNatural}
Let $(y,T)$ be a minimizer of $J$ as in \eqref{funct4}.
Then, for all $x\in[a,T]$, $(y,T)$ is a solution
of the fractional equation
\begin{multline*}
\frac{\partial L}{\partial y}[y](x)+{_xD^\alpha_T}\left(
\frac{\partial L}{\partial v}[y](x) \right)+{_xI_T^\beta}
\left(\frac{\partial L}{\partial w}[y](x)\right)
+\int_x^T \frac{\partial L}{\partial z}[y](t)dt
\cdot \frac{\partial l}{\partial y}\{y\}(x)\\
+{_xD^\alpha_T}\left( \int_x^T \frac{\partial L}{\partial z}[y](t)dt
\cdot \frac{\partial l}{\partial v}\{y\}(x)  \right)
+{_xI^\beta_T}\left( \int_x^T \frac{\partial L}{\partial z}[y](t)dt
\cdot \frac{\partial l}{\partial w}\{y\}(x)  \right)=0
\end{multline*}
and satisfies the transversality conditions
$$
 \left[{_xI_T^{1-\alpha}}\left(  \frac{\partial L}{\partial v}[y](x)
+ \int_x^T  \frac{\partial L}{\partial z}[y](t)\, dt
\cdot \frac{\partial l}{\partial v}\{y\}(x) \right)\right]_{x=a}=0
$$
and
$$
L[y](T)=0.
$$
\end{theorem}

\begin{proof}
Let $h\in \mathcal{F}([a,b];\mathbb{R})$
be a variation, and let $\triangle T$
be a real number. Define the function
$$
j(\epsilon)=J(y+\epsilon h,T+\epsilon \triangle T)
$$
with $|\epsilon| \ll 1$.
Differentiating $j$ at $\epsilon=0$, and using the same procedure
as in Theorem~\ref{ELTEo}, we deduce that
\begin{multline*}
\triangle T \cdot L[y](T)+\int_a^T \left[\frac{\partial L}{\partial y}[y](x)
+{_xD^\alpha_T}\left( \frac{\partial L}{\partial v}[y](x) \right)
+{_xI_T^\beta}\left(\frac{\partial L}{\partial w}[y](x)\right)
+\int_x^T \frac{\partial L}{\partial z}[y](t)dt\cdot
\frac{\partial l}{\partial y}\{y\}(x)\right.\\
\left.+{_xD^\alpha_T}\left( \int_x^T \frac{\partial L}{\partial z}[y](t)dt
\cdot \frac{\partial l}{\partial v}\{y\}(x)\right)
+{_xI^\beta_T}\left( \int_x^T \frac{\partial L}{\partial z}[y](t)dt
\cdot \frac{\partial l}{\partial w}\{y\}(x)  \right)\right]h(x)dx\\
+ \left[ {_xI_T^{1-\alpha}}\left(\frac{\partial L}{\partial v}[y](x) \right)
h(x) \right]_a^T+\left[ {_xI^{1-\alpha}_T}
\left(\int_x^T\frac{\partial L}{\partial z}[y](t)dt
\cdot \frac{\partial l}{\partial v}\{y\}(x) \right) h(x)  \right]_a^T=0.
\end{multline*}
The theorem follows from the arbitrariness of $h$ and $\triangle T$.
\end{proof}

\begin{remark}
If $T$ is fixed, say $T=b$, then $\triangle T=0$
and the transversality conditions reduce to
\begin{equation}
\label{NaturalBoundCond}
\left[{_xI_b^{1-\alpha}}\left( \frac{\partial L}{\partial v}[y](x)
+ \int_x^b \frac{\partial L}{\partial z}[y](t)\, dt \cdot
\frac{\partial l}{\partial v}\{y\}(x) \right)\right]_a=0.
\end{equation}
\end{remark}

\begin{example}
Consider the problem of minimizing the functional $J$ as in \eqref{funcExample},
but without given boundary conditions. Besides equation \eqref{ELeqExample},
extremals must also satisfy
\begin{equation}
\label{boundExample}
\left[{_xI_1^{1-\alpha}}\left( {^C_0D_x^\alpha}y(x)-1 \right)\right]_0=0.
\end{equation}
Again, $\overline y$ given by \eqref{eq:gs:ex} is a solution
of \eqref{ELeqExample} and \eqref{boundExample}.
\end{example}

As a particular case, the following result of
\cite{AGRA1} is deduced.

\begin{corollary}[\textrm{cf.} equations (9) and (12) of \cite{AGRA1}]
\label{Cor:AGRA1}
If $y$ is a minimizer of $J$ as in \eqref{funct7},
then $y$ is a solution of
$$
\frac{\partial L}{\partial y}[y](x)
+{_xD^\alpha_b}\left( \frac{\partial L}{\partial v}[y](x) \right)=0
$$
and satisfies the transversality condition
$$
\left[{_xI_b^{1-\alpha}}\left(\frac{\partial L}{\partial v}[y](x)\right)\right]_a=0.
$$
\end{corollary}

\begin{proof}
The Lagrangian $L$ in \eqref{funct7}
does not depend on ${_aI_x^\beta}y$ and $z$,
and the result follows from Theorem~\ref{TeoNatural}.
\end{proof}

\begin{remark}
\label{new:rem:6}
Observe that the condition
$$
\left[{_xI_b^{1-\alpha}}\left(\frac{\partial L}{\partial v}[y](x)\right)\right]_b=0
$$
is implicitly satisfied in Corollary~\ref{Cor:AGRA1}
(\textrm{cf.} Remark~\ref{remarkIntegral}).
\end{remark}


\section{Fractional isoperimetric problems}
\label{sec:IsoProb}

An isoperimetric problem deals with the question
of optimizing a given functional
under the presence of an integral constraint.
This is a very old question, with its origins in the ancient Greece.
They where interested in determining the shape of a closed curve
with a fixed length and maximum area. This problem is known
as Dido's problem, and is an example of an isoperimetric
problem of the calculus of variations \cite{Brunt}.
For recent advancements on the subject we refer the reader to
\cite{Almeida2,Almeida3,iso:ts,MOMA09} and references therein.
In our case, within the fractional context, we state the isoperimetric
problem in the following way.
Determine the minimizers of a given functional
\begin{equation}
\label{funct2}
J(y)=\int_a^b L(x,y(x),{^C_aD_x^\alpha}y(x),{_aI_x^\beta}y(x),z(x))dx
\end{equation}
subject to the boundary conditions
\begin{equation}
\label{bound2}
y(a)=y_a \quad \mbox{and} \quad y(b)=y_b
\end{equation}
and the fractional integral constraint
\begin{equation}
\label{funct3}
I(y)=\int_a^b G(x,y(x),{^C_aD_x^\alpha}y(x),{_aI_x^\beta}y(x),z(x))dx
=\gamma, \quad \gamma\in\mathbb{R},
\end{equation}
where $z$ is defined by
$$
z(x)=\int_a^x l(t,y(t),{^C_aD_t^\alpha}y(t),{_aI_t^\beta}y(t))dt.
$$
As usual, we assume that all the functions $(x,y,v,w,z)\to L(x,y,v,w,z)$,
$(x,y,v,w)\to l(x,y,v,w)$, and $(x,y,v,w,z)\to G(x,y,v,w,z)$ are of class $C^1$.

\begin{theorem}
\label{IsoPro}
Let $y$ be a minimizer of $J$ as in \eqref{funct2},
under the boundary conditions \eqref{bound2} and isoperimetric constraint \eqref{funct3}.
Suppose that $y$ is not an extremal for $I$ in \eqref{funct3}. Then there exists
a constant $\lambda$ such that $y$ is a solution of the fractional equation
\begin{multline*}
\frac{\partial F}{\partial y}[y](x)
+{_xD^\alpha_b}\left( \frac{\partial F}{\partial v}[y](x) \right)
+{_xI^\beta_b}\left( \frac{\partial F}{\partial w}[y](x) \right)
+\int_x^b \frac{\partial F}{\partial z}[y](t)dt
\cdot \frac{\partial l}{\partial y}\{y\}(x)\\
+{_xD^\alpha_b}\left( \int_x^b \frac{\partial F}{\partial z}[y](t)dt
\cdot \frac{\partial l}{\partial v}\{y\}(x)\right)
+{_xI^\beta_b}\left( \int_x^b \frac{\partial F}{\partial z}[y](t)dt
\cdot \frac{\partial l}{\partial w}\{y\}(x)  \right)=0,
\end{multline*}
where $F=L-\lambda G$, for all $x\in[a,b]$.
\end{theorem}

\begin{proof}
Let $\epsilon_1,\epsilon_2\in \mathbb{R}$ be two real numbers such that
$|\epsilon_1|\ll1$ and $|\epsilon_2|\ll1$, with $\epsilon_1$ free
and $\epsilon_2$ to be determined later,
and let $h_1$ and $h_2$ be two functions satisfying
$$
h_1(a)=h_1(b)=h_2(a)=h_2(b)=0.
$$
Define functions $j$ and $i$ by
$$
j(\epsilon_1,\epsilon_2)=J(y+\epsilon_1h_1+\epsilon_2h_2)
$$
and
$$
i(\epsilon_1,\epsilon_2)=I(y+\epsilon_1h_1+\epsilon_2h_2)-\gamma.
$$
Doing analogous calculations as in the proof of Theorem~\ref{ELTEo}, one has
\begin{multline*}
\left.\frac{\partial i}{\partial \epsilon_2} \right|_{(0,0)}
= \int_a^b\left[\frac{\partial G}{\partial y}[y](x)
+{_xD^\alpha_b}\left( \frac{\partial G}{\partial v}[y](x) \right)
+{_xI^\alpha_b}\left( \frac{\partial G}{\partial w}[y](x) \right)
+\int_x^b \frac{\partial G}{\partial z}[y](t)dt
\cdot \frac{\partial l}{\partial y}\{y\}(x)\right.\\
\left. +{_xD^\alpha_b}\left(
\int_x^b \frac{\partial G}{\partial z}[y](t)dt
\cdot \frac{\partial l}{\partial v}\{y\}(x)  \right)
+{_xI^\beta_b}\left( \int_x^b \frac{\partial G}{\partial z}[y](t)dt
\cdot \frac{\partial l}{\partial w}\{y\}(x)  \right)\right] h_2(x) \, dx.
\end{multline*}
By hypothesis, $y$ is not an extremal for $I$
and therefore there must exist a function $h_2$ for which
$$
\left.\frac{\partial i}{\partial \epsilon_2} \right|_{(0,0)}\not=0.
$$
Since $i(0,0)=0$, by the implicit function theorem there exists
a function $\epsilon_2(\cdot)$, defined
in some neighborhood of zero, such that
\begin{equation}
\label{iso:const:pr}
i(\epsilon_1,\epsilon_2(\epsilon_1))=0.
\end{equation}
On the other hand, $j$ attains a minimum value at $(0,0)$ when subject
to the constraint \eqref{iso:const:pr}.
Because $\nabla i(0,0)\neq (0,0)$,
by the Lagrange multiplier rule \cite[p.~77]{Brunt}
there exists a constant $\lambda$ such that
$$
\nabla(j(0,0)-\lambda i(0,0))=(0,0).
$$
So
$$
\left.\frac{\partial j}{\partial \epsilon_1} \right|_{(0,0)}
-\lambda \left.\frac{\partial i}{\partial \epsilon_1} \right|_{(0,0)}=0.
$$
Differentiating $j$ and $i$ at zero, and
doing the same calculations as before, we get the desired result.
\end{proof}

Using the abnormal Lagrange multiplier rule \cite[p.~82]{Brunt},
the previous result can be generalized to include
the case when the minimizer is an extremal of $I$.

\begin{theorem}
Let $y$ be a minimizer of $J$ as in \eqref{funct2},
subject to the constraints \eqref{bound2} and \eqref{funct3}.
Then there exist two constants $\lambda_0$ and $\lambda$,
not both zero, such that $y$ is a solution of equation
\begin{multline*}
\frac{\partial K}{\partial y}[y](x)+{_xD^\alpha_b}\left( \frac{\partial K}{\partial v}[y](x) \right)
+{_xI^\beta_b} \left( \frac{\partial K}{\partial w}[y](x) \right)
+\int_x^b \frac{\partial K}{\partial z}[y](t)dt\cdot \frac{\partial l}{\partial y}\{y\}(x)\\
+{_xD^\alpha_b}\left( \int_x^b \frac{\partial K}{\partial z}[y](t)dt
\cdot \frac{\partial l}{\partial v}\{y\}(x)\right)
+{_xI^\beta_b}\left( \int_x^b \frac{\partial K}{\partial z}[y](t)dt
\cdot \frac{\partial l}{\partial w}\{y\}(x)  \right)=0
\end{multline*}
for all $x\in[a,b]$, where $K=\lambda_0 L-\lambda G$.
\end{theorem}

\begin{corollary}[\textrm{cf.} Theorem~3.4 of \cite{Almeida}]
Let $y$ be a minimizer of
$$
J(y)=\int_a^b L(x,y(x),{^C_aD_x^\alpha}y(x))dx
$$
subject to the boundary conditions
$$
y(a)=y_a \quad \mbox{and} \quad y(b)=y_b
$$
and the isoperimetric constraint
$$
I(y)=\int_a^b G(x,y(x),{^C_aD_x^\alpha}y(x))dx
=\gamma, \quad \gamma\in\mathbb{R}.
$$
Then, there exist two constants $\lambda_0$ and $\lambda$,
not both zero, such that $y$ is a solution of equation
$$
\frac{\partial K}{\partial y}\left(x,y(x),{^C_aD_x^\alpha}y(x)\right)
+{_xD^\alpha_b}\left( \frac{\partial K}{\partial v}\left(x,y(x),{^C_aD_x^\alpha}y(x)\right) \right) =0
$$
for all $x\in[a,b]$, where $K=\lambda_0 L-\lambda G$. Moreover, if $y$
is not an extremal for $I$, then we may take $\lambda_0=1$.
\end{corollary}


\section{Holonomic constraints}
\label{sec:Holonomic}

In this section we consider the following problem. Minimize the functional
\begin{equation}
\label{funct8}
J(y_1,y_2)=\int_a^b L(x,y_1(x),y_2(x),{^C_aD_x^\alpha}y_1(x),
{^C_aD_x^\alpha}y_2(x),{_aI_x^\beta}y_1(x),{_aI_x^\beta}y_2(x),z(x))dx,
\end{equation}
where $z$ is defined by
$$
z(x)=\int_a^x l(t,y_1(t),y_2(t),{^C_aD_t^\alpha}y_1(t),
{^C_aD_t^\alpha}y_2(t),{_aI_t^\beta}y_1(t),{_aI_t^\beta}y_2(t))dt,
$$
when restricted to the boundary conditions
\begin{equation}
\label{boundconst8}
(y_1(a),y_2(a))=(y_1^a,y_2^a)
\mbox{ and } (y_1(b),y_2(b))=(y_1^b,y_2^b),
\quad y_1^a,y_2^a,y_1^b,y_2^b\in\mathbb{R},
\end{equation}
and the holonomic constraint
\begin{equation}
\label{subsconst}
g(x,y_1(x),y_2(x))=0.
\end{equation}
As usual, here
$$
(x,y_1,y_2,v_1,v_2,w_1,w_2,z)\to L(x,y_1,y_2,v_1,v_2,w_1,w_2,z),
$$
$$
(x,y_1,y_2,v_1,v_2,w_1,w_2)\to l(x,y_1,y_2,v_1,v_2,w_1,w_2)
$$
and
$$
(x,y_1,y_2)\to g(x,y_1,y_2)
$$
are all smooth. In what follows we make use
of the operator $[\cdot,\cdot]$ given by
$$
[y_1,y_2](x)
= (x,y_1(x),y_2(x),{^C_aD_x^\alpha}y_1(x),{^C_aD_x^\alpha}y_2(x),
{_aI_x^\beta}y_1(x),{_aI_x^\beta}y_2(x),z(x))\, ,
$$
we denote
$(x,y_1(x),y_2(x))$ by $(x,\mathbf{y}(x))$,
and the Euler--Lagrange equation obtained in \eqref{ELeq}
with respect to $y_i$ by $(ELE_i)$, $i=1,2$.

\begin{remark}
For simplicity, we are considering functionals depending
only on two functions $y_1$ and $y_2$. Theorem~\ref{thm:hol}
is, however, easily generalized for $n$ variables $y_1,\ldots,y_n$.
\end{remark}

\begin{theorem}
\label{thm:hol}
Let the pair $(y_1,y_2)$ be a minimizer of $J$ as in \eqref{funct8},
subject to the constraints \eqref{boundconst8}--\eqref{subsconst}.
If $\frac{\partial g}{\partial y_2}\not=0$,
then there exists a continuous function $\lambda:[a,b]\to\mathbb{R}$
such that $(y_1,y_2)$ is a solution of
\begin{multline}
\label{ELequation2}
\frac{\partial F}{\partial y_i}[y_1,y_2](x)
+{_xD^\alpha_b}\left( \frac{\partial F}{\partial v_i}[y_1,y_2](x) \right)
+{_xI_b^\beta}\left(\frac{\partial F}{\partial w_i}[y_1,y_2](x)\right)
+\int_x^b \frac{\partial F}{\partial z}[y_1,y_2](t)dt\cdot
\frac{\partial l}{\partial y_i}\{y_1,y_2\}(x)\\
+{_xD^\alpha_b}\left( \int_x^b \frac{\partial F}{\partial z}[y_1,y_2](t)dt
\cdot \frac{\partial l}{\partial v_i}\{y_1,y_2\}(x)  \right)
+{_xI^\beta_b}\left( \int_x^b \frac{\partial F}{\partial z}[y_1,y_2](t)dt
\cdot \frac{\partial l}{\partial w_i}\{y_1,y_2\}(x)\right)=0
\end{multline}
for all $x\in[a,b]$ and $i=1,2$,
where $F[y_1,y_2](x)=L[y_1,y_2](x)-\lambda (x) g(x,\mathbf{y}(x))$.
\end{theorem}

\begin{proof}
Consider a variation of the optimal solution of type
$$
(\overline y_1,\overline y_2)= (y_1+\epsilon h_1,y_2+\epsilon h_2),
$$
where $h_1,h_2$ are two functions defined on $[a,b]$ satisfying
$$
h_1(a)=h_1(b)=h_2(a)=h_2(b)=0,
$$
and $\epsilon$ is a sufficiently small real parameter. Since
$\frac{\partial g}{\partial y_2}(x,\overline{y}_1(x),\overline{y}_2(x))\not=0$
for all $x\in[a,b]$, we can solve equation
$g(x,\overline y_1(x),\overline y_2(x))=0$
with respect to $h_2$, $h_2=h_2(\epsilon,h_1)$. Differentiating
$J(\overline y_1,\overline y_2)$ at $\epsilon=0$, and proceeding
similarly as done in the proof of Theorem~\ref{ELTEo}, we deduce that
\begin{equation}
\label{subsequation}
\int_a^b (ELE_1)h_1(x)+(ELE_2)h_2(x)\, dx=0.
\end{equation}
Besides, since $g(x,\overline y_1(x),\overline y_2(x))=0$,
differentiating at $\epsilon=0$ we get
\begin{equation}
\label{defh2}
h_2(x)=-\frac{\frac{\partial g}{\partial
y_1}(x,\mathbf{y}(x))}{\frac{\partial g}{\partial
y_2}(x,\mathbf{y}(x))}h_1(x).
\end{equation}
Define the function $\lambda$ on $[a,b]$ as
\begin{equation}
\label{subslambda}
\lambda(x)=\frac{(ELE_2)}{\frac{\partial g}{\partial y_2}(x,\mathbf{y}(x))}.
\end{equation}
Combining \eqref{defh2} and \eqref{subslambda},
equation \eqref{subsequation} can be written as
$$
\int_a^b \left[(ELE_1)-\lambda(x) \frac{\partial g}{\partial
y_1}(x,\mathbf{y}(x))\right] h_1(x)\, dx=0.
$$
By the arbitrariness of $h_1$, if follows that
$$
(ELE_1)-\lambda(x) \frac{\partial g}{\partial y_1}(x,\mathbf{y}(x))=0.
$$
Define $F$ as
$$
F[y_1,y_2](x)=L[y_1,y_2](x)-\lambda (x) g(x,\mathbf{y}(x)).
$$
Then, equations \eqref{ELequation2} follow.
\end{proof}


\section{Higher order Caputo derivatives}
\label{sec:Higher}

In this section we consider fractional variational problems,
when in presence of higher order Caputo derivatives.
We will restricted ourselves to the case where the orders
are non integer, since the integer case is already well studied
in the literature (for a modern account see \cite{MyID:194,ferreira,natorres}).

Let $n\in\mathbb{N}$, $\beta > 0$,
and $\alpha_k\in\mathbb{R}$ be such that
$\alpha_k\in(k-1,k)$ for $k\in\{1,\ldots,n\}$.
Admissible functions $y$ belong to $AC^n([a,b];\mathbb{R})$
and are such that ${^C_aD_x^{\alpha_k}}y$, $k = 1, \ldots, n$,
and ${_aI_x^\beta}y$ exist and are continuous on $[a,b]$.
We denote such class of functions by $\mathcal{F}^n([a,b];\mathbb{R})$.
For $\alpha=(\alpha_1,\ldots,\alpha_n)$, define the vector
\begin{equation}
\label{eq:y:ho:l}
{_a^C D _x^\alpha}y(x)
=({_a^C D _x^{\alpha_1}}y(x),\ldots,{_a^C D _x^{\alpha_n}}y(x)).
\end{equation}
The optimization problem is the following: to minimize or maximize
the functional
\begin{equation}
\label{funct5}
J(y)=\int_a^b L(x,y(x),{^C_aD_x^\alpha}y(x),{_aI_x^\beta}y(x),z(x))dx,
\end{equation}
$y \in \mathcal{F}^n([a,b];\mathbb{R})$,
subject to the boundary conditions
\begin{equation}
\label{bound5}
y^{(k)}(a)=y_{a,k} \quad \mbox{and}
\quad y^{(k)}(b)=y_{b,k}, \quad k\in\{0,\ldots,n-1\},
\end{equation}
where $z: [a,b] \to \mathbb{R}$ is defined by
$$
z(x)=\int_a^x l(t,y(t),{^C_aD_t^\alpha}y(t),{_aI_t^\beta}y(t))dt.
$$

\begin{theorem}
\label{thm:16}
If $y \in \mathcal{F}^n([a,b];\mathbb{R})$
is a minimizer of $J$ as in \eqref{funct5},
subject to the boundary conditions \eqref{bound5},
then $y$ is a solution of the fractional equation
\begin{multline*}
\frac{\partial L}{\partial y}[y](x)
+\sum_{k=1}^n{_xD^{\alpha_k}_b}\left( \frac{\partial L}{\partial v_k}[y](x) \right)
+{_xI^\beta_b}\left( \frac{\partial L}{\partial w}[y](x) \right)
+\int_x^b \frac{\partial L}{\partial z}[y](t)dt
\cdot \frac{\partial l}{\partial y}\{y\}(x)\\
+\sum_{k=1}^n{_xD^{\alpha_k}_b}\left( \int_x^b \frac{\partial L}{\partial z}[y](t)dt
\cdot \frac{\partial l}{\partial v_k}\{y\}(x)\right)
+{_xI^\beta_b}\left( \int_x^b \frac{\partial L}{\partial z}[y](t)dt
\cdot \frac{\partial l}{\partial w}\{y\}(x)  \right)=0
\end{multline*}
for all $x\in[a,b]$, where
$[y](x) = \left(x,y(x),{^C_aD_x^\alpha}y(x),{_aI_x^\beta}y(x),z(x)\right)$
with ${^C_aD_x^\alpha}y(x)$ as in \eqref{eq:y:ho:l}.
\end{theorem}

\begin{proof}
Let $h \in \mathcal{F}^n([a,b];\mathbb{R})$
be such that $h^{(k)}(a)=h^{(k)}(b)=0$,
for $k\in\{0,\ldots,n-1\}$. Define the new function
$j$ as $j(\epsilon)=J(y+\epsilon h)$. Then
\begin{multline}
\label{eq1}
\int_a^b \left[ \frac{\partial L}{\partial y}[y](x)h(x)
+ \sum_{k=1}^n \frac{\partial L}{\partial v_k}[y](x){^C_aD^{\alpha_k}_x}h(x)
+\frac{\partial L}{\partial w}[y](x){_aI^\beta_x}h(x)\right.\\
\left.+\frac{\partial L}{\partial z}[y](x)\int_a^x
\left( \frac{\partial l}{\partial y}\{y\}(t)h(t)
+\sum_{k=1}^n\frac{\partial l}{\partial v_k}\{y\}(t){^C_aD^{\alpha_k}_t}h(t)
+\frac{\partial l}{\partial w}\{y\}(t){_aI^\beta_t}h(t) \right)dt\right]dx=0.
\end{multline}
Integrating by parts, we get that
\begin{multline*}
\int_a^b \frac{\partial L}{\partial v_k}[y](x){^C_aD^{\alpha_k}_x}h(x) dx
=\int_a^b {_x D_b^{\alpha_k}}\left(\frac{\partial L}{\partial v_k}[y](x)\right)h(x)dx\\
+\sum_{m=0}^{k-1}\left[{_xD_b^{\alpha_k+m-k}}\left(\frac{\partial L}{\partial v_k}[y](x) \right)
h^{(k-1-m)}(x)\right]_a^b =\int_a^b {_x D_b^{\alpha_k}}\left(
\frac{\partial L}{\partial v_k}[y](x) \right)h(x)dx
\end{multline*}
for all $k\in\{1,\ldots,n\}$. Moreover, one has
$$
\int_a^b \frac{\partial L}{\partial w}[y](x){_aI^\beta_x}h(x) dx
=\int_a^b {_x I_b^\beta} \left(\frac{\partial L}{\partial w}[y](x) \right)h(x)dx,
$$
$$
\int_a^b \frac{\partial L}{\partial z}[y](x)\int_a^x
\frac{\partial l}{\partial y}\{y\}(t)h(t) dt \,  dx
=\int_a^b \left(\int_x^b\frac{\partial L}{\partial z}[y](t)dt
\right) \frac{\partial l}{\partial y}\{y\}(x)h(x) \,  dx,
$$
\begin{align*}
\int_a^b & \frac{\partial L}{\partial z}[y](x)\left(\int_a^x
\frac{\partial l}{\partial v_k}\{y\}(t){^C_aD^{\alpha_k}_t}h(t) dt \right)dx
=\int_a^b \left(\int_x^b\frac{\partial L}{\partial z}[y](t)dt \right)
\frac{\partial l}{\partial v_k}\{y\}(x){^C_aD^{\alpha_k}_x}h(x) \,  dx\\
&= \int_a^b {_xD^{\alpha_k}_b}\left(\int_x^b\frac{\partial L}{\partial z}[y](t)dt
\frac{\partial l}{\partial v_k}\{y\}(x) \right)h(x) \, dx\\
&\qquad +\sum_{m=0}^{k-1}\left[{_xD_b^{\alpha_k+m-k}}\left(
\int_x^b\frac{\partial L}{\partial z}[y](t)dt
\frac{\partial l}{\partial v_k}\{y\}(x)\right)
h^{(k-1-m)}(x)\right]_a^b\\
&= \int_a^b {_xD^{\alpha_k}_b}\left(
\int_x^b\frac{\partial L}{\partial z}[y](t)dt
\frac{\partial l}{\partial v_k}\{y\}(x) \right)h(x) \, dx,
\end{align*}
and
$$
\int_a^b \frac{\partial L}{\partial z}[y](x)\left(\int_a^x
\frac{\partial l}{\partial w}\{y\}(t){_aI^\beta_t}h(t) dt
\right)dx=\int_a^b {_xI^\beta_b}
\left(\int_x^b\frac{\partial L}{\partial z}[y](t)dt
\frac{\partial l}{\partial w}\{y\}(x) \right)h(x) \, dx.
$$
Replacing these last relations into equation \eqref{eq1},
and applying the fundamental lemma of the calculus of variations,
we obtain the intended necessary condition.
\end{proof}

We now consider the higher-order problem without
the presence of boundary conditions \eqref{bound5}.

\begin{theorem}
\label{higherOrderNatural}
If $y \in \mathcal{F}^n([a,b];\mathbb{R})$
is a minimizer of $J$ as in \eqref{funct5},
then $y$ is a solution of the fractional equation
\begin{multline*}
\frac{\partial L}{\partial y}[y](x)
+\sum_{k=1}^n{_xD^{\alpha_k}_b}\left(
\frac{\partial L}{\partial v_k}[y](x) \right)+{_xI^\beta_b}
\left( \frac{\partial L}{\partial w}[y](x) \right)
+\int_x^b \frac{\partial L}{\partial z}[y](t)dt
\cdot \frac{\partial l}{\partial y}\{y\}(x)\\
+\sum_{k=1}^n{_xD^{\alpha_k}_b}\left( \int_x^b
\frac{\partial L}{\partial z}[y](t)dt\cdot \frac{\partial l}{\partial v_k}\{y\}(x)  \right)
+{_xI^\beta_b}\left( \int_x^b \frac{\partial L}{\partial z}[y](t)dt
\cdot \frac{\partial l}{\partial w}\{y\}(x)  \right)=0
\end{multline*}
for all $x\in[a,b]$, and satisfies the natural boundary conditions
\begin{equation}
\label{eq:nbc:ho}
\sum_{m=k}^n \left[ {_xD_b^{\alpha_m-k}}\left(
\frac{\partial L}{\partial v_k}[y](x)
+ \int_x^b\frac{\partial L}{\partial z}[y](t)dt
\frac{\partial l}{\partial v_k}\{y\}(x)\right) \right]_a^b=0,
\quad \mbox{for all} \quad k\in\{1,\ldots,n\}.
\end{equation}
\end{theorem}

\begin{proof}
The proof follows the same pattern as the proof of Theorem~\ref{thm:16}.
Since admissible functions $y$ are not required to satisfy given boundary conditions,
the variation functions $h$ may take any value at the boundaries as well,
and thus the condition
\begin{equation}
\label{HigBoundCons}
h^{(k)}(a)=h^{(k)}(b)=0, \quad \mbox{for } k\in\{0,\ldots,n-1\},
\end{equation}
is no longer imposed \textit{a priori}. If we consider the first variation
of $J$ for variations $h$ satisfying condition \eqref{HigBoundCons},
we obtain the Euler--Lagrange equation. Replacing it on the expression
of the first variation, we conclude that
$$
\sum_{k=1}^n \sum_{m=0}^{k-1}\left[{_xD_b^{\alpha_k
+m-k}}\left(\frac{\partial L}{\partial v_k}[y](x)
+\int_x^b\frac{\partial L}{\partial z}[y](t)dt
\frac{\partial l}{\partial v_k}\{y\}(x)  \right)h^{(k-1-m)}(x)\right]_a^b=0.
$$
To obtain the transversality condition with respect to $k$,
for $k\in\{1,\ldots,n\}$, we consider variations satisfying the condition
$$
h^{(k-1)}(a)\not=0\not=h^{(k-1)}(b) \quad \mbox{and }
h^{(j-1)}(a)=0=h^{(j-1)}(b), \quad \mbox{for all }
j\in\{0,\ldots,n\}\setminus\{k\}.
$$
\end{proof}

\begin{remark}
Some of the terms that appear in the natural boundary
conditions \eqref{eq:nbc:ho} are equal to zero
(\textrm{cf.} Remark~\ref{remarkIntegral} and
Remark~\ref{new:rem:6}).
\end{remark}


\section{Fractional Lagrange problems}
\label{sec:FracOpt}

We now prove a necessary optimality condition
for a fractional Lagrange problem, when the Lagrangian
depends again on an indefinite integral.
Consider the cost functional defined by
\begin{equation}
\label{funct6}
J(y,u)=\int_a^b L\left(x,y(x),u(x),{_aI_x^\beta}y(x),z(x)\right)dx,
\end{equation}
to be minimized or maximized
subject to the fractional dynamical system
\begin{equation}
\label{dynamic6}
{^C_aD_x^\alpha}y(x)=f(x,y(x),u(x),{_aI_x^\beta}y(x),z(x))
\end{equation}
and the boundary conditions
\begin{equation}
\label{bound6}
y(a)=y_a \quad \mbox{and} \quad y(b)=y_b,
\end{equation}
where
$$
z(x)=\int_a^x l\left(t,y(t),{^C_aD_t^\alpha}y(t),{_aI_t^\beta}y(t)\right)dt.
$$
We assume the functions $(x,y,v,w,z)\to f(x,y,v,w,z)$,
$(x,y,v,w,z)\to L(x,y,v,w,z)$, and $(x,y,v,w)\to l(x,y,v,w)$,
to be of class $C^1$ with respect to all their arguments.

\begin{remark}
If $f(x,y(x),u(x),{_aI_x^\beta}y(x),z(x))=u(x)$,
the Lagrange problem \eqref{funct6}--\eqref{bound6}
reduces to the fractional variational problem
\eqref{funct}--\eqref{bound} studied
in Section~\ref{sec:Fundprob}.
\end{remark}

An optimal solution is a pair of functions $(y,u)$ that minimizes $J$
as in \eqref{funct6}, subject to the fractional
dynamic equation \eqref{dynamic6} and the boundary conditions \eqref{bound6}.

\begin{theorem}
If $(y,u)$ is an optimal solution to the fractional
Lagrange problem \eqref{funct6}--\eqref{bound6}, then there exists
a function $p$ for which the triplet $(y,u,p)$ satisfies the Hamiltonian system
$$
\left\{
\begin{array}{ll}
{^C_aD_x^\alpha}y(x)&=\displaystyle\frac{\partial H}{\partial p}\lceil y,u,p \rceil(x),\\
{_xD_b^\alpha}p(x)&=\displaystyle\frac{\partial H}{\partial y}\lceil y,u,p \rceil(x)+{_xI_b^\beta}
\left(\frac{\partial H}{\partial w}
\lceil y,u,p \rceil(x)\right)+\int_x^b \frac{\partial H}{\partial z}\lceil y,u,p \rceil(t)dt\cdot
\frac{\partial l}{\partial y}\{y\}(x)\\
&\quad +\displaystyle{_xD^{\alpha}_b}\left(
\int_x^b \frac{\partial H}{\partial z}\lceil y,u,p \rceil(t)dt\cdot
\frac{\partial l}{\partial v}\{y\}(x)\right)+{_xI^{\beta}_b}\left(
\int_x^b \frac{\partial H}{\partial z}\lceil y,u,p \rceil(t)dt\cdot
\frac{\partial l}{\partial w}\{y\}(x)\right)
\end{array}
\right.$$
and the stationary condition
$$
\frac{\partial H}{\partial u}\lceil y,u,p \rceil(x)=0,
$$
where the Hamiltonian $H$ is defined by
$$
H\lceil y,u,p \rceil(x)=L(x,y(x),u(x),{_aI_x^\beta}y(x),z(x))
+p(x)f(x,y(x),u(x),{_aI_x^\beta}y(x),z(x))
$$
and
$$
\lceil y,u,p \rceil(x)= (x,y(x),u(x),{_aI_x^\beta}y(x),z(x),p(x))\, ,
\quad \{y\}(x)=(x,y(x),{^C_aD_x^\alpha}y(x),{_aI_x^\beta}y(x)).
$$
\end{theorem}

\begin{proof}
The result follows applying Theorem~\ref{ELeqMult} to
$$
{J^*}(y,u,p)=\int_a^b  H\lceil y,u,p \rceil(x)
-p(x) {^C_aD_x^\alpha}y(x) dx
$$
with respect to $y$, $u$ and $p$.
\end{proof}

In the particular case when $L$ does not depend
on ${_aI_x^\beta}y$ and $z$,
we obtain \cite[Theorem~3.5]{Gastao0}.

\begin{corollary}[Theorem~3.5 of \cite{Gastao0}]
Let $(y(x),u(x))$ be a solution of
$$
J(y,u)=\int_a^b L(x,y(x),u(x))dx \longrightarrow \min
$$
subject to the fractional control system
${^C_aD_x^\alpha}y(x)=f(x,y(x),u(x))$
and the boundary conditions $y(a)=y_a$ and $y(b)=y_b$.
Define the Hamiltonian by
$H(x,y,u,p)=L(x,y,u) + p f(x,y,u)$.
Then there exists a function $p$ for which
the triplet $(y,u,p)$ fulfill the Hamiltonian system
$$
\begin{cases}
{^C_aD_x^\alpha}y(x)=\frac{\partial H}{\partial p}(x,y(x),u(x),p(x)),\\
{_xD_b^\alpha}p(x)=\frac{\partial H}{\partial y}(x,y(x),u(x),p(x)),
\end{cases}
$$
and the stationary condition
$\frac{\partial H}{\partial u}(x,y(x),u(x),p(x))=0$.
\end{corollary}


\section{Sufficient conditions of optimality}
\label{sec:SufConditions}

To begin, let us recall the notions
of convexity and concavity for $C^1$ functions of several variables.

\begin{definition}
Given $k\in\{1,\ldots,n\}$ and a function
$\Psi:D\subseteq \mathbb{R}^n\to \mathbb{R}$
such that $\partial \Psi / \partial x_i$ exist and
are continuous for all $i\in\{k,\ldots,n\}$,
we say that $\Psi$ is convex (concave) in $(x_k,\ldots,x_n)$ if
\begin{multline*}
\Psi(x_1+\tau_1,\ldots,x_{k-1}+\tau_{k-1},x_k+\tau_k,
\ldots,x_n+\tau_n)-\Psi(x_1,\ldots,x_{k-1},x_k,\ldots,x_n)\\
\geq \, (\leq) \, \frac{\partial \Psi}{\partial x_k}(x_1,\ldots,x_{k-1},x_k,\ldots,x_n)\tau_k
+ \cdots + \frac{\partial \Psi}{\partial x_n}(x_1,\ldots,x_{k-1},x_k,\ldots,x_n)\tau_n
\end{multline*}
for all $(x_1,\ldots,x_n),(x_1+\tau_1,\ldots,x_n+\tau_n)\in D$.
\end{definition}

\begin{theorem}
Consider the functional $J$ as in \eqref{funct},
and let $y \in \mathcal{F}([a,b];\mathbb{R})$
be a solution of the fractional
Euler--Lagrange equation \eqref{ELeq}
satisfying the boundary conditions \eqref{bound}.
Assume that $L$ is convex in $(y,v,w,z)$.
If one of the two following conditions is satisfied,
\begin{enumerate}
\item $l$ is convex in $(y,v,w)$ and
$\frac{\partial L}{\partial z}[y](x) \geq 0$ for all $x \in [a,b]$;
\item $l$ is concave in $(y,v,w)$ and
$\frac{\partial L}{\partial z}[y](x) \leq 0$ for all $x \in [a,b]$;
\end{enumerate}
then $y$ is a (global) minimizer of problem \eqref{funct}--\eqref{bound}.
\end{theorem}

\begin{proof}
Consider $h$ of class $\mathcal{F}([a,b];\mathbb{R})$
such that $h(a)=h(b)=0$. Then,
\begin{equation*}
\begin{split}
J(y+h) & - J(y) = \int_a^b
L\Biggl(x,y(x)+h(x),{^C_aD_x^\alpha}y(x)
+{^C_aD_x^\alpha}h(x),{_aI_x^\beta}y(x)+{_aI_x^\beta}h(x),\\
&\qquad\qquad\qquad\qquad \int_a^x l(t,y(t)+h(t),{^C_aD_t^\alpha}y(t)
+{^C_aD_t^\alpha}h(t),{_aI_t^\beta}y(t)+{_aI_t^\beta}h(t))dt \Biggr)dx\\
& \quad -\int_a^b L(x,y(x),{^C_aD_x^\alpha}y(x),
{_aI_x^\beta}y(x),\int_a^x l(t,y(t),{^C_aD_t^\alpha}y(t),{_aI_t^\beta}y(t))dt)dx\\
&\geq \int_a^b \left[ \frac{\partial L}{\partial y}[y](x)h(x)
+\frac{\partial L}{\partial v}[y](x){^C_aD^\alpha_x}h(x)
+\frac{\partial L}{\partial w}[y](x){_aI^\beta_x}h(x)\right.\\
&\quad \left. +\frac{\partial L}{\partial z}[y](x)
\int_a^x\left( \frac{\partial l}{\partial y}\{y\}(t)h(t)
+\frac{\partial l}{\partial v}\{y\}(t){^C_aD^\alpha_t}h(t)
+\frac{\partial l}{\partial w}\{y\}(t){_aI^\beta_t}h(t)\right)dt\right]dx\\
&= \int_a^b \left[ \frac{\partial L}{\partial y}[y](x)
+{_xD^\alpha_b}\left( \frac{\partial L}{\partial v}[y](x) \right)
+{_xI_b^\beta}\left(\frac{\partial L}{\partial w}[y](x)\right)
+\int_x^b \frac{\partial L}{\partial z}[y](t)dt
\cdot \frac{\partial l}{\partial y}\{y\}(x)\right.\\
&\quad \left. +{_xD^\alpha_b}\left(
\int_x^b \frac{\partial L}{\partial z}[y](t)dt
\cdot \frac{\partial l}{\partial v}\{y\}(x)  \right)
+{_xI^\beta_b}\left( \int_x^b \frac{\partial L}{\partial z}[y](t)dt
\cdot \frac{\partial l}{\partial w}\{y\}(x)  \right)\right] h(x)dx
= 0.
\end{split}
\end{equation*}
\end{proof}

One can easily include the case when the
boundary conditions \eqref{bound} are not given.

\begin{theorem}
Consider functional $J$ as in \eqref{funct}
and let $y \in \mathcal{F}([a,b];\mathbb{R})$
be a solution of the fractional Euler--Lagrange
equation \eqref{ELeq} and the fractional natural boundary condition
\eqref{NaturalBoundCond}. Assume
that $L$ is convex in $(y,v,w,z)$. If one
of the two next conditions is satisfied,
\begin{enumerate}
\item $l$ is convex in $(y,v,w)$ and
$\frac{\partial L}{\partial z}[y](x) \geq 0$ for all $x \in [a,b]$;
\item $l$ is concave in $(y,v,w)$ and
$\frac{\partial L}{\partial z}[y](x) \leq 0$ for all $x \in [a,b]$;
\end{enumerate}
then $y$ is a (global) minimizer of \eqref{funct}.
\end{theorem}


\section{Numerical simulations}
\label{sec:NumSim}

Solving a variational problem usually means solving
Euler--Lagrange differential equations
subject to some boundary conditions. It turns out that
most fractional Euler--Lagrange equations cannot be solved analytically.
Therefore, in practical terms, numerical methods need to be developed and used
in order to solve the fractional variational problems.
A numerical scheme to solve fractional Lagrange problems
has been presented in \cite{AgrawalNum}. The method is based on approximating
the problem to a set of algebraic equations using some basis functions.
A more general approach can be found in \cite{Tricaud} that uses
the Oustaloup recursive approximation of the fractional derivative,
and reduces the problem to an integer order (classical) optimal control problem.
A similar approach is presented in \cite{Jelicic}, using an expansion
formula for the left Riemann--Liouville fractional derivative
developed in \cite{Atan}. Here we use a modified approximation
(see Remark~\ref{rem:ns}) based on the same expansion,
to reduce a given fractional problem to a classical one.
The expansion formula is given in the following lemma.

\begin{lemma}[\textrm{cf.} equation (12) of \cite{Atan}]
Suppose that $f\in AC^2[0,b]$, ${f''}\in L_1[0,b]$ and $0<\alpha\leq 1$.
Then the left Riemann--Liouville fractional derivative can be expanded as
\begin{equation*}
{_0D_x^\alpha} f(x)=A(\alpha)x^{-\alpha}f(x)+B(\alpha)x^{1-\alpha}{f'}(x)
-\sum_{k=2}^{\infty}C(k,\alpha)x^{1-k-\alpha}v_k(x),
\end{equation*}
where
\begin{eqnarray*}
{v'}_k(x)&=&(1-k)x^{k-2}f(x),\qquad v_k(0)=0, \qquad k=2,3,\ldots,\\
A(\alpha)&=&\frac{1}{\Gamma(1-\alpha)}
-\frac{1}{\Gamma(2-\alpha)\Gamma(\alpha-1)}
\sum_{k=2}^{\infty}\frac{\Gamma(k-1+\alpha)}{(k-1)!},\\
B(\alpha)&=&\frac{1}{\Gamma(2-\alpha)}\left[1
+\sum_{k=1}^{\infty}\frac{\Gamma(k-1+\alpha)}{\Gamma(\alpha-1)k!}\right],\\
C(k,\alpha)&=&\frac{1}{\Gamma(2-\alpha)\Gamma(\alpha-1)}\frac{\Gamma(k-1+\alpha)}{(k-1)!}.
\end{eqnarray*}
\end{lemma}

In practice, we only keep a finite number of terms
in the series. We use the approximation
\begin{equation}
\label{expanMom}
{_0D_x^\alpha} f(x)\simeq A(\alpha,N)x^{-\alpha}f(x)
+B(\alpha,N)x^{1-\alpha}{f'}(x)-\sum_{k=2}^{N}C(k,\alpha)x^{1-k-\alpha}v_k(x)
\end{equation}
for some fixed number $N$, where
\begin{eqnarray*}
A(\alpha,N)&=&\frac{1}{\Gamma(1-\alpha)}
-\frac{1}{\Gamma(2-\alpha)\Gamma(\alpha-1)}
\sum_{k=2}^{N}\frac{\Gamma(k-1+\alpha)}{(k-1)!},\\
B(\alpha,N)&=&\frac{1}{\Gamma(2-\alpha)}\left[1
+\sum_{k=1}^{N}\frac{\Gamma(k-1+\alpha)}{\Gamma(\alpha-1)k!}\right].
\end{eqnarray*}

\begin{remark}
\label{rem:ns}
In \cite{Atan} the authors use the fact that
$1+\sum_{k=1}^{\infty}\frac{\Gamma(k-1+\alpha)}{\Gamma(\alpha-1)k!}=0$,
and apply in their method the approximation
\begin{equation*}
\label{expanAtan}
{_0D_x^\alpha} f(x)\simeq A(\alpha,N)x^{-\alpha}f(x)
-\sum_{k=2}^NC(k,\alpha)x^{1-k-\alpha}v_k(x).
\end{equation*}
Regarding the value of $B(\alpha,N)$ for some values of $N$ (see Table~\ref{tab}),
we take the first derivative in the expansion into account
and keep the approximation in the form of equation \eqref{expanMom}.
\begin{table}[h!]
\center
\begin{tabular}{|c|c c c c c c c|}
\hline
$N$         &    4   &    7   &   15   &   30   &   70   &  120   &  170   \\
\hline
$B(0.3,N)$  & 0.1357 & 0.0928 & 0.0549 & 0.0339 & 0.0188 & 0.0129 & 0.0101 \\
$B(0.5,N)$  & 0.3085 & 0.2364 & 0.1630 & 0.1157 & 0.0760 & 0.0581 & 0.0488 \\
$B(0.7,N)$  & 0.5519 & 0.4717 & 0.3783 & 0.3083 & 0.2396 & 0.2040 & 0.1838 \\
$B(0.9,N)$  & 0.8470 & 0.8046 & 0.7481 & 0.6990 & 0.6428 & 0.6092 & 0.5884 \\
\hline
\end{tabular}
\caption{Values of $B(\alpha,N)$ for
$\alpha\in\{0.3,0.5,0.7,0.9\}$ and different values of $N$.}
\label{tab}
\end{table}
\end{remark}

We illustrate with Examples~\ref{example:bra} and \ref{example1}
how the approximation \eqref{expanMom} provides an accurate
and efficient numerical method to solve fractional variational problems.

\begin{example}
\label{ex:new:sec6}
We obtain an approximated solution to the problem
considered in Example~\ref{example:bra}.
Since $y(0)=0$, the Caputo derivative coincides
with the Riemann--Liouville derivative
and we can approximate the fractional
problem using \eqref{expanMom}.
We reformulate the problem using the Hamiltonian formalism by
letting $^C_0D_x^{\alpha}y(x)=u(x)$. Then,
\begin{equation}
\label{eq:h:1}
A(\alpha,N)x^{-\alpha}y(x)+B(\alpha,N)x^{1-\alpha}{y'}(x)
-\sum_{k=2}^{N}C(k,\alpha)x^{1-k-\alpha}v_k(x)=u(x).
\end{equation}
We also include the variable $z(x)$ with
$$
{z'}(x)=\left( y(x)-x^{\alpha+1}\right)^2.
$$
In summary, one has the following Lagrange problem:
\begin{equation}
\label{AppExample2}
\begin{gathered}
\tilde{J}(y) = \int_0^1 [(u(x)-\Gamma(\alpha+2)x)^2+z(x)]dx \longrightarrow \min \\
\begin{cases}
{y'}(x) = -AB^{-1}x^{-1}y(x)+\sum_{k=2}^{N}B^{-1}C_kx^{-k}v_k(x)+B^{-1}x^{\alpha-1}u(x)\\
{v'}_k(x) = (1-k)x^{k-2}y(x),\qquad k=1,2,\ldots\\
{z'}(x) = \left( y(x)-x^{\alpha+1}\right)^2
\end{cases}
\end{gathered}
\end{equation}
subject to the boundary conditions $y(0)=0$, $z(0)=0$ and $v_k(0)=0$, $k=1,2,\ldots$
Setting $N=2$, the Hamiltonian is given by
\begin{multline*}
H=-[(u(x)-\Gamma(\alpha+2)x)^2+z(x)]+p_1(x)\left(-AB^{-1}x^{-1}y(x)+B^{-1}C_2x^{-2}v_2(x)
+B^{-1}x^{\alpha-1}u(x)\right)\\
-p_2(x)y(x)+p_3(x)\left( y(x)-x^{\alpha+1}\right)^2.
\end{multline*}
Using the classical necessary optimality condition for problem \eqref{AppExample2},
we end up with the following two point boundary value problem:
\begin{equation}
\label{sys2}
\left\{
\begin{array}{ll}
{y'}(x)   & =-AB^{-1}x^{-1}y(x)+ B^{-1}C_2x^{-2}v_2(x)
+\frac{1}{2}B^{-2}x^{2\alpha-2}p_1(x)+\Gamma(\alpha+2)B^{-1}x^{\alpha}\\
{v'_2}(x) & =-y(x)\\
{z'}(x)   & =(y(x)-x^{\alpha+1})^2\\
{p'_1}(x) & =AB^{-1}x^{-1}p_1(x)+p_2(x)-2p_3(x)(y(x)-x^{\alpha+1})\\
{p'_2}(x) & =-B^{-1}C_2x^{-2}p_1(x)\\
{p'_3}(x) & = 1
\end{array}
\right.
\end{equation}
subject to the boundary conditions
\begin{equation}
\label{sysB2}
\begin{cases}
y(0)=0 \\
v_2(0)=0\\
z(0)=0
\end{cases}
\qquad
\begin{cases}
y(1)=1\\
p_2(1)=0\\
p_3(1)=0.
\end{cases}
\end{equation}
We solved system \eqref{sys2} subject to \eqref{sysB2}
using the MATLAB$^\circledR$ built-in function \texttt{bvp4c}.
The resulting graph for $y(x)$, together with the corresponding
value of $J$, is given in Figure~\ref{Fig2}.
\begin{figure}[!ht]
\begin{center}
\includegraphics[scale=.8]{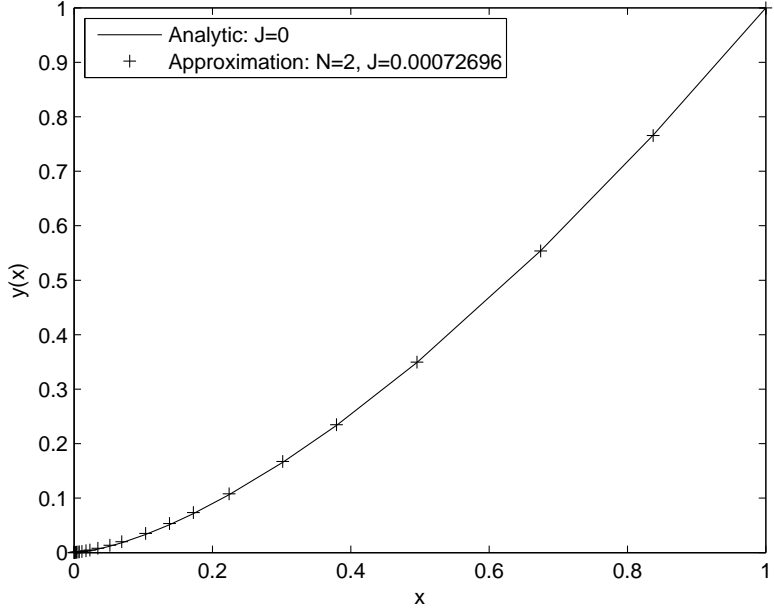}
\caption{Analytic vs numerical solution
to problem of Example~\ref{example:bra}.}\label{Fig2}
\end{center}
\end{figure}
\end{example}

Our numerical method works well, even in the case
the minimizer is not a Lipschitz function.

\begin{example}
An approximated solution to the problem
considered in Example~\ref{example1} can be obtained
following exactly the same steps as in Example~\ref{ex:new:sec6}.
Recall that the minimizer \eqref{eq:gs:ex} to that
problem is not a Lipschitz function.
As before, one has $y(0)=0$ and the Caputo derivative coincides
with the Riemann--Liouville derivative. We approximate the fractional
problem using \eqref{expanMom}.
Let $^C_0D_x^{\alpha}y(x)=u(x)$. Then \eqref{eq:h:1} holds.
In this case the variable $z(x)$ satisfies
$$
{z'}(x)=\left( y(x)-\frac{x^\alpha}{\Gamma(\alpha+1)}\right)^2
$$
and we approximate the fractional variational problem
with the following classical one:
\begin{equation*}
\begin{gathered}
\tilde{J}(y) = \int_0^1 [(u(x)-1)^2+z(x)]dx \longrightarrow \min \\
\begin{cases}
{y'}(x) = -AB^{-1}x^{-1}y(x)+\sum_{k=2}^{N}B^{-1}C_kx^{-k}v_k(x)+B^{-1}x^{\alpha-1}u(x)\\
{v'}_k(x) = (1-k)x^{k-2}y(x),\qquad k=1,2,\ldots\\
{z'}(x) = \left( y(x)-\frac{x^\alpha}{\Gamma(\alpha+1)}\right)^2
\end{cases}
\end{gathered}
\end{equation*}
subject to the boundary conditions $y(0)=0$, $z(0)=0$ and $v_k(0)=0$, $k=1,2,\ldots$
Setting $N=2$, the Hamiltonian is given by
\begin{multline*}
H=-[(u(x)-1)^2+z(x)]+p_1(x)\left(-AB^{-1}x^{-1}y(x)+B^{-1}C_2x^{-2}v_2(x)
+B^{-1}x^{\alpha-1}u(x)\right)\\
-p_2(x)y(x)+p_3(x)\left( y(x)-\frac{x^\alpha}{\Gamma(\alpha+1)}\right)^2.
\end{multline*}
The classical theory \cite{MR0166037} tell us to solve the system
\begin{equation}
\label{sys}
\left\{
\begin{array}{ll}
{y'}(x)   & =-AB^{-1}x^{-1}y(x)+ B^{-1}C_2x^{-2}v_2(x)
+\frac{1}{2}B^{-2}x^{2\alpha-2}p_1(x)+B^{-1}x^{\alpha-1}\\
{v'_2}(x) & =-y(x)\\
{z'}(x)   & =(y(x)-\frac{x^\alpha}{\Gamma(\alpha+1)})^2\\
{p'_1}(x) & =AB^{-1}x^{-1}p_1(x)+p_2(x)-2p_3(x)(y(x)
-\frac{x^\alpha}{\Gamma(\alpha+1)})\\
{p'_2}(x) & =-B^{-1}C_2x^{-2}p_1(x)\\
{p'_3}(x) & = 1
\end{array}
\right.
\end{equation}
subject to boundary conditions
\begin{equation}
\label{sysB}
\begin{cases}
y(0)=0 \\
v_2(0)=0\\
z(0)=0
\end{cases}
\qquad
\begin{cases}
y(1)=\frac{1}{\Gamma(\alpha+1)}\\
p_2(1)=0\\
p_3(1)=0.
\end{cases}
\end{equation}
As done in Example~\ref{ex:new:sec6}, we solved \eqref{sys}--\eqref{sysB}
using the MATLAB$^\circledR$ built-in function \texttt{bvp4c}.
The resulting graph for $y(x)$, together with the corresponding
value of $J$, is given in Figure~\ref{Fig} in contrast with
the exact minimizer \eqref{eq:gs:ex}.
\begin{figure}[!ht]
\begin{center}
\includegraphics[scale=.8]{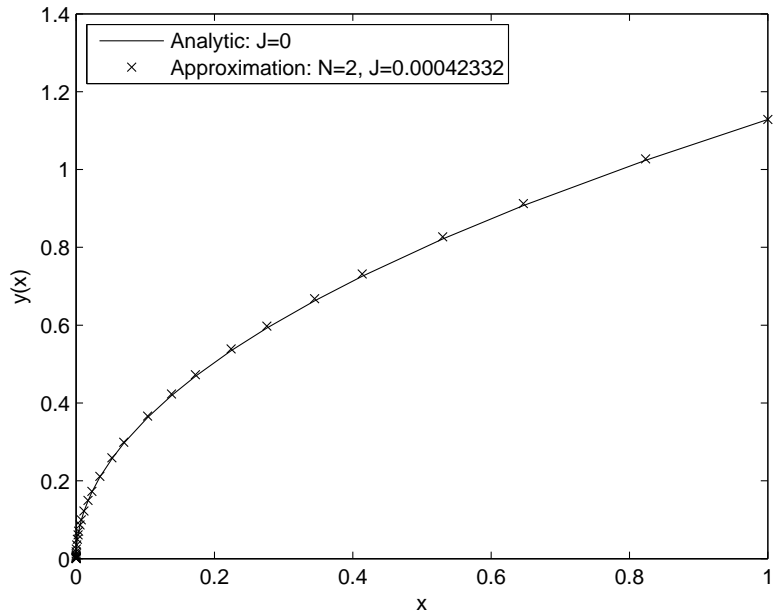}
\caption{Analytic vs numerical solution
to problem of Example~\ref{example1}.}\label{Fig}
\end{center}
\end{figure}
\end{example}


\section*{Acknowledgments}

Work supported by the \emph{Portuguese Foundation for Science and Technology} (FCT),
through the {\it Center for Research and Development in Mathematics and Applications}
(CIDMA) and the Ph.D. fellowship SFRH/BD/33761/2009 (Shakoor Pooseh).
The authors are very grateful to a referee
for valuable remarks and comments, which
significantly contributed to the quality of the paper.




{\small
\begin{thebibliography}{99}

\bibitem{AgrawalNum}
O. P. Agrawal,
A general formulation and solution scheme
for fractional optimal control problems,
Nonlinear Dynam. {\bf 38} (2004), no.~1-4, 323--337.

\bibitem{AGRA1}
O. P. Agrawal,
Generalized Euler-Lagrange equations and transversality
conditions for FVPs in terms of the Caputo derivative,
J. Vib. Control {\bf 13} (2007), no.~9-10, 1217--1237.

\bibitem{MR2345467}
O. P. Agrawal,
Fractional variational calculus in terms of Riesz fractional derivatives,
J. Phys. A {\bf 40} (2007), no.~24, 6287--6303.

\bibitem{MyID:182}
R. Almeida, A. B. Malinowska\ and\ D. F. M. Torres,
A fractional calculus of variations for multiple integrals
with application to vibrating string,
J. Math. Phys. 51 (2010), no.~3, 033503, 12 pp.
{\tt arXiv:1001.2722}

\bibitem{Ankara:Ric}
R. Almeida, A. B. Malinowska\ and\ D. F. M. Torres,
Fractional Euler-Lagrange differential equations via Caputo derivatives,
in {\it Fractional Dynamics and Control}
(eds: D. Baleanu, J. A. Tenreiro Machado, and A. Luo),
Springer, in press.

\bibitem{Almeida2}
R. Almeida\ and\ D. F. M. Torres,
H\"olderian variational problems subject to integral constraints,
J. Math. Anal. Appl. {\bf 359} (2009), no.~2, 674--681.
{\tt arXiv:0807.3076}

\bibitem{Almeida3}
R. Almeida\ and\ D. F. M. Torres,
Isoperimetric problems on time scales with nabla derivatives,
J. Vib. Control {\bf 15} (2009), no.~6, 951--958.
{\tt arXiv:0811.3650}

\bibitem{Almeida1}
R. Almeida\ and\ D. F. M. Torres,
Calculus of variations with fractional
derivatives and fractional integrals,
Appl. Math. Lett. {\bf 22} (2009), no.~12, 1816--1820.
{\tt arXiv:0907.1024}

\bibitem{MyID:154}
R. Almeida and D. F. M. Torres,
Leitmann's direct method for fractional optimization problems,
Appl. Math. Comput. 217 (2010), no.~3, 956--962.
{\tt arXiv:1003.3088}

\bibitem{Almeida}
R. Almeida\ and\ D. F. M. Torres,
Necessary and sufficient conditions for the fractional calculus
of variations with Caputo derivatives,
Commun. Nonlinear Sci. Numer. Simulat. {\bf 16} (2011), no.~3, 1490--1500.
{\tt arXiv:1007.2937}

\bibitem{Atan}
T. M. Atanackovic\ and\ B. Stankovic,
On a numerical scheme for solving differential
equations of fractional order,
Mech. Res. Comm. {\bf 35} (2008), no.~7, 429--438.

\bibitem{MyID:179}
N. R. O. Bastos, R. A. C. Ferreira\ and\ D. F. M. Torres,
Discrete-time fractional variational problems,
Signal Processing {\bf 91} (2011), no.~3, 513--524.
{\tt arXiv:1005.0252}

\bibitem{MyID:194}
A. M. C. Brito da Cruz, N. Martins\ and\ D. F. M. Torres,
Higher-order Hahn's quantum variational calculus,
Nonlinear Anal. (2011), in press.
DOI: 10.1016/j.na.2011.01.015
{\tt arXiv:1101.3653}

\bibitem{El-Nabulsi1}
R. A. El-Nabulsi\ and\ D. F. M. Torres,
Necessary optimality conditions for fractional action-like integrals
of variational calculus with Riemann-Liouville derivatives of order
$(\alpha,\beta)$, Math. Methods Appl. Sci. {\bf 30} (2007), no.~15, 1931--1939.
{\tt arXiv:math-ph/0702099}

\bibitem{El-Nabulsi2}
R. A. El-Nabulsi\ and\ D. F. M. Torres,
Fractional actionlike variational problems,
J. Math. Phys. {\bf 49} (2008), no.~5, 053521, 7 pp.
{\tt arXiv:0804.4500}

\bibitem{ferreira}
R. A. C. Ferreira\ and\ D. F. M. Torres,
Higher-order calculus of variations on time scales,
in {\it Mathematical control theory and finance}
(eds: A. Sarychev, A. Shiryaev, M. Guerra, and M. do R. Grossinho),
149--159, Springer, Berlin, 2008.
{\tt arXiv:0706.3141}

\bibitem{iso:ts}
R. A. C. Ferreira\ and\ D. F. M. Torres,
Isoperimetric problems of the calculus of variations on time scales,
in {\it Nonlinear Analysis and Optimization II}
(eds: A.~Leizarowitz, B.~S.~Mordukhovich, I.~Shafrir, and A.~J.~Zaslavski),
Contemporary Mathematics, vol.~514, Amer. Math. Soc., Providence, RI, 2010, pp.~123--131.
{\tt arXiv:0805.0278}

\bibitem{MyID:191}
R. A. C. Ferreira\ and\ D. F. M. Torres,
Fractional $h$-difference equations arising from the calculus of variations,
Appl. Anal. Discrete Math. {\bf 5} (2011), in press.
DOI: 10.2298/AADM110131002F
{\tt arXiv:1101.5904}

\bibitem{fraser}
C. G. Fraser,
Isoperimetric problems in the variational
calculus of Euler and Lagrange,
Historia Math. {\bf 19} (1992), no.~1, 4--23.

\bibitem{gastao:delfim}
G. S. F. Frederico\ and\ D. F. M. Torres,
A formulation of Noether's theorem for fractional problems
of the calculus of variations,
J. Math. Anal. Appl. {\bf 334} (2007), no.~2, 834--846.
{\tt arXiv:math/0701187}

\bibitem{gasta1}
G. S. F. Frederico\ and\ D. F. M. Torres,
Fractional conservation laws in optimal control theory,
Nonlinear Dynam. {\bf 53} (2008), no.~3, 215--222.
{\tt arXiv:0711.0609}

\bibitem{Gastao0}
G. S. F. Frederico\ and\ D. F. M. Torres,
Fractional optimal control in the sense of Caputo
and the fractional Noether's theorem,
Int. Math. Forum {\bf 3} (2008), no.~9-12, 479--493.
{\tt arXiv:0712.1844}

\bibitem{Gastao}
G. S. F. Frederico\ and\ D. F. M. Torres,
Fractional Noether's theorem in the Riesz-Caputo sense,
Appl. Math. Comput. {\bf 217} (2010), no.~3, 1023--1033.
{\tt arXiv:1001.4507}

\bibitem{Gregory}
J. Gregory,
Generalizing variational theory to include
the indefinite integral, higher derivatives,
and a variety of means as cost variables,
Methods Appl. Anal. {\bf 15} (2008), no.~4, 427--435.

\bibitem{Jelicic}
Z. D. Jelicic\ and\ N. Petrovacki,
Optimality conditions and a solution scheme
for fractional optimal control problems,
Struct. Multidiscip. Optim. {\bf 38} (2009), no.~6, 571--581.

\bibitem{Kilbas}
A. A. Kilbas, H. M. Srivastava\ and\ J. J. Trujillo,
{\it Theory and applications of fractional differential equations},
North-Holland Mathematics Studies, 204, Elsevier, Amsterdam, 2006.

\bibitem{MOMA09}
A. B. Malinowska\ and\ D. F. M. Torres,
Delta-nabla isoperimetric problems,
Int. J. Open Probl. Comput. Sci. Math. {\bf 3} (2010), no.~4, 124--137.
{\tt arXiv:1010.2956}

\bibitem{Malinowska}
A. B. Malinowska\ and\ D. F. M. Torres,
Generalized natural boundary conditions for fractional
variational problems in terms of the Caputo derivative,
Comput. Math. Appl. {\bf 59} (2010), no.~9, 3110--3116.
{\tt arXiv:1002.3790}

\bibitem{natorres}
N. Martins\ and\ D. F. M. Torres,
Calculus of variations on time scales with nabla derivatives,
Nonlinear Anal. {\bf 71} (2009), no.~12, e763--e773.
{\tt arXiv:0807.2596}

\bibitem{Nat}
N. Martins\ and\ D. F. M. Torres,
Generalizing the variational theory on time scales
to include the delta indefinite integral, submitted.

\bibitem{Miller}
K. S. Miller\ and\ B. Ross,
{\it An introduction to the fractional calculus
and fractional differential equations},
A Wiley-Interscience Publication, Wiley, New York, 1993.

\bibitem{MyID:163}
D. Mozyrska\ and\ D. F. M. Torres,
Minimal modified energy control for fractional
linear control systems with the Caputo derivative,
Carpathian J. Math. {\bf 26} (2010), no.~2, 210--221.
{\tt arXiv:1004.3113}

\bibitem{MyID:181}
D. Mozyrska\ and\ D. F. M. Torres,
Modified optimal energy and initial memory
of fractional continuous-time linear systems,
Signal Process. {\bf 91} (2011), no.~3, 379--385.
{\tt arXiv:1007.3946}

\bibitem{MyID:207}
T. Odzijewicz, A. B. Malinowska\ and\ D. F. M. Torres,
Fractional variational calculus with classical and combined Caputo derivatives,
Nonlinear Anal. (2011), in press.
DOI: 10.1016/j.na.2011.01.010
{\tt arXiv:1101.2932}

\bibitem{MyID:203}
T. Odzijewicz\ and\ D. F. M. Torres,
Fractional calculus of variations for double integrals,
Balkan J. Geom. Appl. {\bf 16} (2011), in press.
{\tt arXiv:1102.1337}

\bibitem{MR0166037}
L. S. Pontryagin, V. G. Boltyanskii, R. V. Gamkrelidze\ and\ E. F. Mishchenko,
{\it The mathematical theory of optimal processes},
Translated from the Russian by K. N. Trirogoff;
edited by L. W. Neustadt Interscience Publishers
John Wiley \& Sons, Inc.\, New York, 1962.

\bibitem{CD:Riewe:1996}
F. Riewe,
Nonconservative Lagrangian and Hamiltonian mechanics,
Phys. Rev. E (3) {\bf 53} (1996), no.~2, 1890--1899.

\bibitem{CD:Riewe:1997}
F. Riewe,
Mechanics with fractional derivatives,
Phys. Rev. E (3) {\bf 55} (1997), no.~3, part B, 3581--3592.

\bibitem{samko}
S. G. Samko, A. A. Kilbas\ and\ O. I. Marichev,
{\it Fractional integrals and derivatives},
Translated from the 1987 Russian original,
Gordon and Breach, Yverdon, 1993.

\bibitem{Tricaud}
C. Tricaud\ and\ Y. Chen,
An approximate method for numerically solving fractional
order optimal control problems of general form,
Comput. Math. Appl. {\bf 59} (2010), no.~5, 1644--1655.

\bibitem{Brunt}
B. van Brunt,
{\it The calculus of variations},
Universitext, Springer, New York, 2004.

\end{thebibliography}}
\end{document}